\documentclass[12pt]{amsart}
\usepackage[latin1]{inputenc}   
\usepackage{amsmath,amsthm,amssymb,inputenc,euscript,graphicx,psfrag,xcolor,pstricks,enumerate}
\title[Schr\"odinger equation on homogeneous trees]
{Schr\"odinger equation\\
on homogeneous trees}

\author{Alaa Jamal Eddine}

\address{}

\email{alaa.jamal-eddine@etu.univ-orleans.fr}

\date{\today}

\subjclass[2000] {Primary 35Q55, 43A90;\\
\hspace{3cm}Secondary 22E35, 43A85, 81Q05, 81Q35, 35R02}


\newcommand\1{1\hskip-1mm\text{\rm I}}

\renewcommand\Im{\operatorname{Im}}
\newcommand{\N}{\mathbb{N}}
\renewcommand\Re{\operatorname{Re}}
\newcommand{\R}{\mathbb{R}}
\newcommand\ssb{\hskip-.25mm}
\newcommand\ssf{\hskip.25mm}

\newtheorem{theoreme}{Theorem}[section]
\newtheorem{propositione}[theoreme]{Proposition}
\newtheorem{lemmae}[theoreme]{Lemma}
\newtheorem{corollarye}[theoreme]{Corollary}

\newtheorem{defne}{Definition}[section] 

\newcommand{\pp}{<}
\newcommand{\pg}{>}
\newcommand{\dsps}{\displaystyle}
 
\begin{document}
%
\begin{abstract}
Let $\mathbb{T}$ be a homogeneous tree and $\mathcal{L}$ the Laplace operator on $\mathbb{T}$. We consider the semilinear Schr\"odinger equation associated to $\mathcal{L}$ with a power-like nonlinearity $F$ of degree $\gamma$. We first obtain dispersive estimates and Strichartz estimates with no admissibility conditions. We next deduce global well-posedness for small $L^2$ data with no gauge invariance assumption on the nonlinearity $F$. On the other hand if $F$ is gauge invariant,  $L^2$ conservation leads to global well-posedness for arbitrary $L^2$ data. Notice that, in contrast with the Euclidean case, these global well-posedness results hold for all finite $\gamma\pg1$. 
We finally prove scattering for arbitrary $L^2$ data under the gauge invariance assumption.
\end{abstract}
\maketitle

\section{Introduction}
In this paper we consider the semilinear Schr\"odinger equation  
\begin{equation}\label{semilin}
\begin{cases}
i\partial_t u +\mathcal{L}u=F(u)\\
u(0)=f
\end{cases}
\end{equation}
associated with the positive laplacian $\mathcal{L}$ on homogeneous trees $\mathbb{T}$ of degree $Q+1\geq 3.$
The essential tools for the study of (\ref{semilin}) are  dispersive and Strichartz type estimates. In the Euclidean case, (\ref{semilin}) has been considered for large classes of nonlinearities (see \cite{tao2}, \cite{caz} and the references therein).
In this case, the dispersive estimate   
$$\left\|e^{it\Delta}f\right\|_{L^\infty(\mathbb{R}^n)}\le C\,|t|^{-n/2}\left\|f\right\|_{L^1(\mathbb{R}^n)}$$ 
holds for the homogeneous problem. A well known procedure (introduced by Kato \cite{k}, developed by Ginibre \& Velo \cite{gv}
and perfected by Keel \& Tao \cite{kt}) then leads to the following Strichartz estimates
\begin{equation}\label{stric}
\left\|u\right\|_{L^p(I,\,L^q(\mathbb{R}^n))}\le C\left\|f\right\|_{L^2(\mathbb{R}^n)}+C\left\|F\right\|_{L^{\tilde{p}^{\prime}}(I,\,L^{\tilde{q}^{\prime}}(\mathbb{R}^n))}
\end{equation}
for the linear problem
\begin{equation}\label{lin prob}
\begin{cases}
i\partial_t u(t,x) +\Delta u(t,x)= F(t,x)\,,\\
u(0,x)=f(x).
\end{cases}
\end{equation} 
These estimates hold for all bounded or unbounded time interval $I$ and for all couples $(p,q)$, $(\tilde{p},\tilde{q})$ $\in [2,\infty] \times [2,\infty)$ satisfying the admissibility condition 
\begin{equation}\label{admis R}
\frac{2}{p}+\frac{n}{q}=\frac{n}{2}.
\end{equation}
Notice that both endpoints $(p,q)=(\infty,2)$ and $(p,q)=(2,\frac{2n}{n-2})$ are included in dimension $n\geq 3$ while only the first one is included in dimension $n=2$.
In view of their important application to nonlinear problems, many attempts have been made to study the dispersive properties of $(\ref{semilin})$ on various Riemannian manifolds (see \cite{ap}, \cite{apv}, \cite{b}, \cite{bgt}, \cite{bo}, \cite{gp}, \cite{is}, \cite{p2}).
More precisely, dispersive and Strichartz estimates for the Schr\"odinger equation on real hyperbolic spaces $\mathbb{H}^n$, which are manifolds with constant negative curvature, have been obtained by Banica, Anker and Pierfelice, Ionescu and Staffilani (\cite{b}, \cite{ap}, \cite{is}). In \cite{ap} and \cite{is}, the authors have obtained sharp dispersive and Strichartz estimates for solutions to the homogeneous and inhomogeneous problems with no radial assumption on the initial data and for a wider range of couples than in the Euclidean case.
More precisely they have obtained $(\ref{stric})$ for admissible couples in the range

\begin{equation}\label{admissib H}
\{(\frac{1}{p},\frac{1}{q})\in (0,\frac{1}{2}]\times (0,\frac{1}{2}) ; \frac{2}{p}+\frac{n}{q}\geq\frac{n}{2}\}\cup\{(0,\frac{1}{2})\}.
\end{equation}

In this paper we consider the Schr\"odinger equations on homogeneous trees $\mathbb{T}$, which are discrete analogs of 
hyperbolic spaces and more precisely 0-hyperbolic spaces according to Gromov.
In \cite{s} A. Setti has already investigated the $L^p-L^q$ mapping properties of the complex time heat operator $e^{\tau\mathcal{L}}$ on $\mathbb{T}$ for $\Re \tau \geq 0$. His study is based on a careful kernel analysis, using the Abel transform and reducing this way to $\mathbb{Z}$.
Our paper is organized as follows.
In Section 2, we recall the structure of homogeneous trees and spherical harmonic analysis thereon. In Section 3, we resume the analysis of the Schr\"odinger kernel and deduce our main two estimates, that we now state.\\
\;
\textbf{Dispersive Estimate}~: Let $2\pp q\le\infty$. Then 
\begin{equation*}
\left\|e^{it\mathcal{L}}\right\|_{L^{{q}^{\prime}}(\mathbb{T})\to L^{q}(\mathbb{T})}\lesssim
\begin{cases}
1 \qquad\hspace{0.2cm} \textrm{if}\;|t|<1,
\\|t|^{-\frac{3}{2}} \quad \textrm{if}\;|t|\geq 1.
\end{cases}
\end{equation*}
Notice that in the limit case $q=2$, we have $L^2$ conservation $\left\|e^{-it\mathcal{L}}f\right\|_{L^2(\mathbb{T})}= \left\|f\right\|_{L^2(\mathbb{T})}$ for all $t\in\mathbb{R}$.\\ 
\;
\textbf{Strichartz estimates}~:\, Assume that 
$(\frac{1}{p},\frac{1}{q})$ and $(\frac{1}{\tilde{p}},\frac{1}{\tilde{q}})$ belong to the square 
\begin{equation*}\label{admissible square}
\mathcal{C}=(0,\frac{1}{2}]\times[0,\frac{1}{2})\cup\left\{(0,\frac{1}{2})\right\}.
\end{equation*}
Then the solution to the linear problem $(\ref{lin prob})$ satisfy
$$\left\|u(t,x)\right\|_{L^{\infty}_{t}L^{2}_{x}}+\left\|u(t,x)\right\|_{L^{p}_{t}L^{q}_{x}}\lesssim\left\|f(x)\right\|_{L^{2}_{x}}+\left\|F(t,x)\right\|_{L^{\tilde{p}^{\prime}}_{t}L^{\tilde{q}^{\prime}}_{x}}.$$
\\
Notice that the set of admissible couples for $\mathbb{T}$ is much wider than the corresponding set $(\ref{admissib H})$ for real hyperbolic spaces $\mathbb{H}^n$ which was itself wider than the admissible set $(\ref{admis R})$. This striking result may be regarded as an effect of strong dispersion in hyperbolic geometry, combined with the absence of local obstructions.
\begin{figure}[ht]\label{square}
\begin{center}
{\includegraphics[width=6cm, height=4cm]{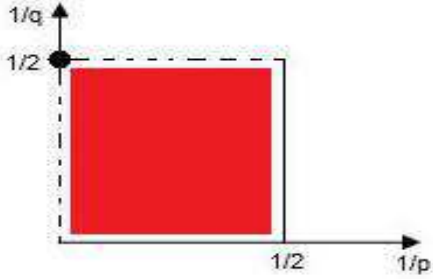}}\quad
\caption{Admissible set for $\mathbb{T}$}
\end{center}
\end{figure}
\\
In Section 4, we use these estimates to prove well-posedness and scattering results for the nonlinear Cauchy problem $(\ref{semilin}).$
We will deal with power-like nonlinearities in the following sense:
 there exist constants $\gamma\pg 1$ and $C\pg 0$ such that
\begin{equation}\label{nonlinearity}
\begin{cases}
|F(u)|\le C|u|^{\gamma},\\
|F(u)-F(v)|\le C(|u|^{\gamma-1}+|v|^{\gamma-1})|u-v|.
\end{cases}
\end{equation}

Recall that $F$ is said to be \textit{gauge invariant} if 
\begin{equation}\label{gauge}
\Im\{F(u)\overline{u}\}=0
\end{equation}
and that $L^2$ conservation holds in this case~:
\begin{equation}\label{conservation}
\forall t,\;\int_\mathbb{T}|u(t,x)|^2 d\mu (x)=\int_\mathbb{T}|f(x)|^2 d\mu(x).
\end{equation}
Here is a typical example of such a nonlinearity~: 
\begin{equation}\label{typical}
F(u)=\lambda|u|^{\gamma-1}u\hspace{0.8cm} \textrm{with}\, \lambda\in\mathbb{R},\,\gamma\pg 1.
\end{equation}
In this particular case, another quantity is conserved, namely the energy 
\begin{equation}\label{energy}
E(t)=\frac{1}{4}\dsps\sum_{x,y\; d(x,y)=1}|u(t,x)-u(t,y)|^2+\lambda\, \frac{Q+1}{\gamma+1}\dsps\sum_{x\in\mathbb{T}}|u(t,x)|^{\gamma+1}.
\end{equation}
Notice that $E(t)\pg 0$ in the so-called defocusing case $\lambda\pg 0$. Here are our well-posedness results.\\
\newline
\textbf{Well-posedness}. Assume $1\pp\gamma\pp\infty$. Then the NLS $(\ref{semilin})$ is globally well-posed for small $L^2$ data. It is locally well-posed for arbitrary $L^{2}$ data. Moreover, in the gauge invariance case $(\ref{gauge})$, local solutions extend to global ones.\\
$\,$

Notice that there is no restriction on the power $\gamma\pg 1$ contrarily to the Euclidean case and even to the hyperbolic space case. Recall on one hand that, on $\mathbb{R}^n$, global  well-posedness for small $L^2$ data is known to hold for $\gamma= 1+\frac{4}{n}$ and local well-posedness for arbitrary $L^2$ data if $1\pp\gamma\pp 1+\frac{4}{n}$ (see \cite{caz}). Recall on the other hand that better results hold on $\mathbb{H}^n$, namely  global well-posedness for small $L^2$ data  in the whole range $1\pp\gamma\le\ 1+\frac{4}{n}$, and $L^2$ local well-posedness for arbitrary data in the range $1\pp\gamma\pp 1+\frac{4}{n}$ (see \cite{ap}). In both cases, $L^2$ local well-posedness for arbitrary data extends  who $L^2$ global well-posedness when the nonlinearity $F$ satisfies the gauge invariance property $(\ref{gauge})$.
Here are our scattering results.\\
\;
\textbf{Scattering}\,. Assume $1\pp\gamma\pp\infty$. Then under the gauge invariance condition $(\ref{gauge})$, for any data $f\in L^2$, the unique global solution $u(t,x)$ to the NLS $(\ref{semilin})$ scatters to linear solutions at infinity, i.e there exist $u_\pm\in L^2$ such that
$$\left\|u(t)-e^{it\mathcal{L}}u_\pm\right\|_{L^{2}(\mathbb{T})}\to 0 \;\textrm{as}\;t\to\pm\infty.$$
\\
Notice again the absence of restriction on $\gamma\pg 1.$ Thich is in sharp contrast with $\mathbb{R}^n,$ where scattering for small $L^2$ data holds in the critical case $\gamma=1+\frac{4}{n}$ but may fail for $1\pp\gamma\pp 1+\frac{4}{n}$ (see \cite{caz}). Our result is also better than on $\mathbb{H}^n,$ where scattering for small $L^2$ data holds in the range $1\pp\gamma\le 1+\frac{4}{n}$ (see \cite{ap}).
\section{Homogeneous Trees }
A homogeneous tree of degree $Q+1$ is an infinite connected graph $\mathbb{T}$ with no loops , in which every vertex is adjacent to $Q+1$ other vertices. We shall identify $\mathbb{T}$ with its set of vertices. $\mathbb{T}$ carries a natural distance d and a natural measure $\mu$. Specifically, $d(x,y)$ is the number of edges of the shortest path joining $x$ to $y$ and $\mu$ is the couting measure. $L^{p}(\mathbb{T})$ denotes the associated Lebesgue space, whose norm is given by
\begin{center}
\begin{displaymath}
\left\|f\right\|_{L^{p}(\mathbb{T})}=
\left\{ \begin{array}{ll}
\Bigl(\dsps\int_{\mathbb{T}}|f(x)|^{p}d\mu\Bigr)^{\frac{1}{p}}=\,\Bigl(\dsps\sum_{x\in\mathbb{T}}|f(x)|^{p}\Bigr)^{\frac{1}{p}}& \textrm{if $1\leq p<\infty$,}
\\ \mathop{\sup}_{x\in\mathbb{T}}|f(x)|& \textrm{if $p=\infty$}
\end{array}\right.
\end{displaymath}
\end{center}
There is no Sobolev spaces theory on $\mathbb{T}$. For instance, one might define the Sobolev space $H^{1}(\mathbb{T})$ by 
$$H^{1}(\mathbb{T})=:\left\{f\in L^{2}(\mathcal{V})|\,|\nabla|f\in L^{2}(\mathcal{E})\right\}$$
where $\mathcal V$ denotes the set of vertices, $\mathcal E$ the set of edges and
$$|\nabla|f(\overline{xy})=:\,|f(y)-f(x)|,$$
is the ``positive gradient''. As 
\begin{equation*}
\left\|\nabla f\right\|^{2}_{L^2(\mathcal{E})}=\;\frac{1}{2}\dsps\sum_{x,y\in \mathcal{V}\, d(x,y)=1} |f(x)-f(y)|^2\le 2(Q+1)\dsps\sum_{x\in\mathcal{V}}|f(x)|^2=\,2(Q+1)\left\|f\right\|^2_{L^2(\mathcal{V})}.
\end{equation*}
we see that  $H^1(\mathbb{T})=L^{2}(\mathbb{T}).$

We fix a base point o and we denote by $|x|=d(x,o)$ the distance to o. Functions depending only on $|x|$ are called radial. If $E(\mathbb{T})$ is a space of functions defined on $\mathbb{T},$ $E(\mathbb{T})^{\sharp}$ denotes the subspace of radial functions in $E(\mathbb{T}).$\\
\;
Let $G$ be the group of isometries of the metric space $\mathbb{T}$ and $K$ the stabiliser of o in $G$.
\begin{propositione}$($Tits \cite{tits}$)$ For every finite subset $\mathcal{F}$ of $\mathcal{V}$, denote by $Aut_{\mathcal{F}}(\mathbb{T})$ the group of automorphisms $g$ of $\mathbb{T}$ such that $g(x)=x$ for all $x\in \mathcal{F}.$ Then G is equipped with a topology of locally compact totally discontinuous group such that the subgroups $Aut_{\mathcal{F}}(\mathbb{T})$ form a fundamental system of neighborhoods of the identity in $G$. Moreover a subgroup H is  maximal, open, compact in $G$ if and only if H is the stabiliser of a point $x$ in $\mathcal{V}.$
\end{propositione} 
Thus $G$ is a locally compact, totally discontinuous, unimodular group, $K$ is an open compact subgroup of $G$, and $(G,K)$ is a Gelfand pair.\\
\newline\indent
The natural action of $G$ on $\mathbb{T}$, which is transitive and continuous, allows us to identify $\mathbb{T}$ with $G/K.$ Moreover $K$ acts transitively on spheres centered at 0. We shall identify functions defined on $\mathbb{T}$ with right $K-$ invariant functions on $G$, and radial functions on $\mathbb{T}$ with $K-$bi-invariant ones on $G$.\\
\indent
Let us recall some basic tools of harmonic analysis on $\mathbb{T}$ (see \cite{ftn}, \cite{ca}, \cite{cms2} for more details). We normalize the Haar measure on $G$ in such a way that $K$ has a unit mass. Then
$$\dsps\sum_{x\in\mathbb{T}}f(x)=\dsps\int_{G}f(g.o)\,dg$$
for all $f\in L^{1}(\mathbb{T}),$ and this allows us to define the convolution of two functions on $\mathbb{T}$ by
\begin{equation}\label{convolution}
f_1*f_2(g.o)=\dsps\int_{G}f_1(h.o)f_2(h^{-1}g.o)\,dh\,\,\;\forall g\in G.
\end{equation}
If $f_2$ is radial, then $(\ref{convolution})$ rewrites
$$f_1*f_2(x)=\dsps\sum_{n\geq 0}f_{2}(n)\dsps\sum_{y\in S(x,n)}f_1(y),$$
where $S(x,n)$ denotes the sphere with center $x$ and radius $n$. 
Notice that 
$$|S(x,n)|=
\begin{cases}
1\qquad\hspace{1.6cm}\textrm{if}\;n=0,\\
(Q+1)Q^{n-1}\quad\textrm{if}\;n\geq 1.
\end{cases}
$$
An interesting property of this convolution product is the following version of the Kunze-Stein phenomenon \cite{ks}, due to Nebbia \cite{N} (see also \cite{co}).

\begin{propositione}
For all $1\le r\pp 2$, we have
\begin{equation}\label{ks}
L^{2}(\mathbb{T})*L^{r}(\mathbb{T})^{\sharp}\subset L^{2}(\mathbb{T}).
\end{equation}
By such an inclusion, we mean that there exists a constant $C_r\pg 0$ such that
$$\left\|f_1*f_2\right\|_{L^2}\le C_r\left\|f_1\right\|_{L^2}\left\|f_2\right\|_{L^r}\hspace{0.6cm}
\forall f_1\in L^2(\mathbb{T}),\,\forall f_2\in L^r(\mathbb{T})^{\sharp}.
$$
\end{propositione}

\begin{corollarye}\label{lq estimate}
For all $2\pp q,r\pp\infty,$ such that $\frac{q}{2}\pp r\pp q,$ one has
\begin{equation}\label{dualks}
L^{q^{\prime}}(\mathbb{T})*L^{r}(\mathbb{T})^{\sharp}\subset L^{q}(\mathbb{T}).
\end{equation}
\end{corollarye}
The proof is obtained by interpolation between the dual version of $(\ref{ks})$ and the trivial inclusion $L^{1}(\mathbb{T})*L^{1}(\mathbb{T})^\sharp\subset L^{1}(\mathbb{T}).$\\
The combinatorial Laplacian on $\mathbb{T}$ is defined by 
$$\mathcal{L}=I-\mathcal{M}$$
where $I$ denotes the identity map and $\mathcal{M}$ the mean operator 
\begin{equation}\label{mean}
\mathcal{M}f(x)=\frac{1}{Q+1}\dsps\sum_{y\in S(x,1)}f(y).
\end{equation}
Notice that $\mathcal{L}$ is a convolution operator associated to a radial function
$$\mathcal{L}f=f*(\delta_o-\nu),$$
where $\delta_o$ denotes the Dirac measure at o and $\nu$ the normalized uniform measure on $S(x,1)$.\\
Let us next recall the main ingredients of spherical analysis on $\mathbb{T}$. Set $\tau=\frac{2\pi}{\log Q}$ , and consider the holomorphic function
$$\gamma(\lambda)=\frac{Q^{i\lambda}+Q^{-i\lambda}}{Q^{\frac{1}{2}}+Q^{-\frac{1}{2}}}=\frac{2}{Q^{\frac{1}{2}}+Q^{-\frac{1}{2}}}\cos\lambda(\log Q)=\gamma(0) \cos \frac{2\pi}{\tau}\lambda.$$
The spherical function $\varphi_{\lambda}$ of index $\lambda\in\mathbb{C}$ is the unique radial eigenfunction of $\mathcal{M},$ which is associated to the eigenvalue $\gamma(\lambda)$ and which is normalized by $\varphi_\lambda(0)=1.$ Here is an explicit expression~:
\begin{displaymath}
\varphi_{\lambda}(n)= \begin{cases}
\mathbf{c}(\lambda)\,Q^{(-\frac{1}{2}+i\lambda)n}+c(-\lambda)\,Q^{(-\frac{1}{2}-i\lambda)n} \quad \textrm{if $\lambda\notin(\frac{\tau}{2})\,\mathbb{Z}$,}

\\(-1)^{jn}\;(1+\frac{Q^{\frac{1}{2}}-Q^{-\frac{1}{2}}}{Q^{\frac{1}{2}}+Q^{-\frac{1}{2}}}n)Q^{-\frac{n}{2}} \quad \hspace{1cm}\textrm{if $\lambda\in(\frac{\tau}{2})j\,\mathbb{Z},$}
\end{cases}
\end{displaymath}
where $\mathbf{c}$ is the meromorphe function given by
$$\mathbf{c}(z)=\frac{1}{Q^\frac{1}{2}+Q^{-\frac{1}{2}}}\,\frac{Q^{\frac{1}{2}+iz}-Q^{-\frac{1}{2}-iz}}{Q^{iz}-Q^{-iz}}\;\;\forall z\in\mathbb{C}\backslash (\frac{\tau}{2})\,\mathbb{Z}.$$
\\
Notice the symmetries
 \begin{equation}
\varphi_{-\lambda}=\varphi_\lambda\;\textrm{and}\;\varphi_{\lambda+\tau}=\varphi_{\lambda}\quad\forall\,\lambda\in\mathbb{C}.
\end{equation}
\\
The spherical Fourier transform  of a radial function $f$ on $\mathbb{T}$, let say with finite support, is then defined by the formula
$$\mathcal{H}f(\lambda)=\dsps\sum_{x\in\mathbb{T}}f(x)\,\varphi_\lambda(x)=f(0)+\dsps\sum_{n\geq 1}(1+Q)\,Q^{n-1}\,f(n)\,\varphi_{\lambda}(n)\;\;\;\forall\lambda\in\mathbb{C}.$$
 By the above symmetries of the spherical functions, $\mathcal{H}f$ is even and $\tau$-periodic. The following inverse and Plancherel formulae hold~:
\begin{equation*}
f(n)=\frac{Q^{\frac{1}{2}}}{Q^{\frac{1}{2}}+Q^{-\frac{1}{2}}}\,\,\frac{1}{\tau}\int_{0}^{\frac{\tau}{2}}\frac{d\lambda}{|c(\lambda)|^{2}}\,\,\mathcal{H}f(\lambda)\,\,\varphi_{\lambda}(n),
\end{equation*}
and
\begin{equation*}
\left\|f\right\|_{2}^{2}=\frac{Q^{\frac{1}{2}}}{Q^{\frac{1}{2}}+Q^{-\frac{1}{2}}}\,\,\frac{1}{\tau}\dsps\left(\int_{0}^{\frac{\tau}{2}}\frac{d\lambda}{|c(\lambda)|^{2}}\,\,\dsps\left|\mathcal{H}f(\lambda)\right|^{2}\dsps\right)\,\,\,\,\,\,\,\;\forall f\in L^{2}(\mathbb{T})^{\sharp}.
\end{equation*}
The Abel transform of a radial function $f$ on  $\mathbb{T}$ is defined by 
\begin{equation}\label{abel}
\mathcal{A}f(n)=Q^{\frac{|n|}{2}}f(|n|)+\frac{Q^{\frac{1}{2}}-Q^{-\frac{1}{2}}}{Q^{\frac{1}{2}}}\dsps\sum_{k\geq 1}Q^{\frac{|n|}{2}+k}f(|n|+2k).
\end{equation}
Then $\mathcal{A}f$ is even and
\begin{equation}\label{fourier}
\mathcal{H}=\mathcal{F}\circ\mathcal{A}
\end{equation}
where
\begin{equation}\label{fourier Z}
\mathcal{F}g(\lambda)=\dsps\sum_{n\in\mathbb{Z}}g(\lambda)\,Q^{in\lambda}=\,g(0)+2\dsps\sum_{n=1}^{\infty}g(\lambda)\,\cos(\frac{2\pi}{\tau}\lambda\,n).
\end{equation}
Hence 
\begin{equation}\label{fourier inverse}
\mathcal{H}^{-1}=\mathcal{A}^{-1}\circ\mathcal{F}^{-1},
\end{equation}
where the Abel transform $(\ref{abel})$ is inverted by
$$
f(n)=\dsps\sum_{k=0}^{\infty}Q^{-\frac{n}{2}-k}\left\{\mathcal{A}f(n+2k)-\mathcal{A}f(n+2k+2)\right\}\, ,
$$
and the Fourier transform $(\ref{fourier Z})$ by
\begin{equation*}
g(n)=\frac{1}{\tau}\int_{-\frac{\tau}{2}}^{\frac{\tau}{2}}d\lambda\,\,\mathcal{F}g(\lambda)\,Q^{-i\lambda\,n}=\frac{1}{\pi}\int_{0}^{\pi}d\lambda\;\mathcal{F}g(\frac{\tau}{2\pi}\lambda)\;\cos(\lambda\,n)\, .
\end{equation*}

\section{Dispersive and Strichartz estimates}
$\,$
\vspace{0.7cm}
Consider first the homogeneous linear equation on $\mathbb{T}$~:
\begin{equation}\label{shl}
\begin{cases}
i\partial_{t}u(t,x)+ \mathcal{L}u(t,x)=0,\\
u(0,x)=f(x)~,
\end{cases}
\end{equation}
whose solution is given by 
$$u(t,x)=e^{it\mathcal{L}}f(x)=f*s_t(x).$$
Here $s_t(x)$ is the radial convolution kernel associated to the Schr\"odinger operator $e^{it\mathcal{L}},$ whose spherical Fourier transform is given by $$\mathcal{H}s_t(\lambda)=e^{it[1-\gamma(\lambda)]}.
$$
The following expression is easily deduced from $(\ref{fourier inverse})$~:
\begin{equation}\label{noyau}
s_t(n)=e^{it}\frac{2}{\pi}\dsps\sum_{k\geq 0}Q^{-\frac{n}{2}-k}\int_{0}^{\pi}d\lambda\,e^{-i\gamma(0)t\cos\lambda}\sin\lambda\,\sin\lambda(n+2k+1).
\end{equation}
\noindent
\begin{propositione}\label{pointwise}
The following pointwise kernel estimates hold,
uniformly in \,$t\!\in\!\R$ \ssf and \,$n\!\in\!\N$\,{\rm:}
\begin{center}
{\rm(22)}\hfill
$|s_{\ssf t}(n)|\lesssim\begin{cases}
\,Q^{-\frac n2}
&\text{if \,}|t|\!<\!1\ssf,\\
\,|t|^{-\frac32}\,(1\ssb+\ssb n)^{2}\,Q^{-\frac n2}
&\text{if \,}|t|\!\ge\!1\ssf.
\end{cases}$\hfill${}$
\end{center}
\end{propositione}
The estimate is easily seen to hold for $|t|\pp 1$. For $|t|\ge 1$, we integrate by parts $(\ref{noyau})$ and apply the next lemma to the resulting expression 
\begin{equation*}
 s_t(n)=\frac{2i}{\pi\,\gamma(0)}\,\frac{e^{it}}{t^{-1}}\,\sum_{k\ge 0}Q^{-\frac{n}{2}-k} \,(n+2k+1)\,\int_{0}^{\pi}d\lambda\,e^{-i\gamma(0)t\cos(\lambda)}\,\cos(n+2k+1)\lambda.
\end{equation*}

\noindent
\begin{lemmae}\label{jtree}
Consider the integral
\begin{equation}\label{OI}
J(t,m)=\int_{\,0}^{\ssf\pi}\!d\lambda\;
e^{\,i\ssf c\ssf t\cos\lambda}\,\cos{m}\ssf\lambda\,,
\end{equation}
where \,$c\!>\!0$ \ssf is a fixed constant.
Then there exists a constant \,$C\!>\!0$
\ssf such that
\begin{equation*}
|J(t,m)|\le C\,|t|^{-\frac12}\,{(1\!+\ssb m)}
\end{equation*}
for every \,${m\!\in\!\N}$ \ssf
and for every \,$t\!\in\!\R$ \ssf with \,$|t|\!\ge\!1$\ssf.
\end{lemmae}
\begin{proof}[Proof]
This estimate is obtained either
by expressing \eqref{OI} in terms of Bessel functions\,:
\begin{equation*}
J(t,m)=\pi\,i^{\ssf m}\ssf J_m(c\,t)
\end{equation*}
(see for instance \cite[(10.9.2)]{DLMF}),
and by using classical estimates for these functions,
or by analyzing the oscillatory integral \eqref{OI}
as in \cite[Section 7.1]{Stein}.
More precisely,
let $\chi$ be a smooth bump function around the origin such that
\begin{equation*}
\sum\nolimits_{\ssf\ell\in\mathbb{Z}}
\chi\ssf(\lambda\!-\ssb\ell\ssf\tfrac\pi2)=1
\qquad\forall\,\lambda\!\in\!\mathbb{R}
\end{equation*}
 and let us split up
\begin{equation*}
J(t,m)=\ssf\sum\nolimits_{\ssf\ell=0}^{\,2}\ssf\underbrace{
\int_{\,0}^{\ssf\pi}\!d\lambda\,\chi(\lambda\!-\ssb\ell\ssf\tfrac\pi2)\,
e^{\,i\ssf c\ssf t\cos\lambda}\ssf\cos{m}\ssf\lambda
}_{J_\ell(t,m)}\,.
\end{equation*}
On one hand, we obtain
\begin{equation*}
|J_1(t,m)|\le C\,|t|^{-1}\ssf(1\!+\ssb m)\,,
\end{equation*}
after performing an integration by parts based on
\begin{equation*}
e^{\,i\ssf c\ssf t\cos\lambda}
=\tfrac i{c\,t\ssf \sin\lambda}\,
\tfrac{\partial}{\partial\lambda}
\bigl(e^{\,i\ssf c\ssf t\cos\lambda}\bigr)\,.
\end{equation*}
On the other hand,
by applying \cite[Corollary p.\;334]{Stein},
we obtain
\begin{equation*}
|J_0(t,m)|\le C\,|t|^{-\frac12}\ssf(1\!+\ssb m)
\quad\text{and}\quad
|J_2(t,m)|\le C\,|t|^{-\frac12}\ssf(1\!+\ssb m)\,.
\end{equation*}
\end{proof}
\begin{corollarye}\label{stlq}
For any $q\pg 2$, the following kernel estimate hold~:
\begin{equation*}
\left\|s_{t}\right\|_{L^{q}(\mathbb{T})}\lesssim\begin{cases}
1 \hspace{2mm}\qquad\textrm{if}\;\;|t|<1
\\|t|^{-\frac{3}{2}} \quad \textrm{if}\;\;|t|\geq 1.
\end{cases}
\end{equation*}
\end{corollarye}
\begin{proof}[Proof]
The case $q=\infty$ follows immediately from Proposition $\ref{pointwise}$. Assume that $2\pp q\pp\infty.$ Then, as  $s_t$ is a radial kernel, we have
$$\left\|s_{t}\right\|^q_{L^{q}(\mathbb{T})}= |s_t(0)|^q+(Q+1)\dsps\sum_{n=1}^{\infty}Q^{n-1}\,|s_{t}(n)|^{q}.$$
We conclude by using Proposition $\ref{pointwise}.$ On one hand,
if $|t|\le 1$
$$\left\|s_{t}\right\|^{q}_{L^{q}}\lesssim 1+(Q+1)\dsps\sum_{n=1}^{\infty}Q^{n-1}\,Q^{-\frac{nq}{2}}\asymp\dsps\sum_{n=0}^{\infty}Q^{-n(\frac{q}{2}-1)}\lesssim 1$$
On the other hand, if $|t|\geq 1,$
\begin{align*}
\left\|s_{t}\right\|_{L^{q}}^{q}&\lesssim|t|^{-\frac{3q}{2}}[1+(Q+1)\dsps\sum_{n=1}^{\infty}Q^{n-1}\,(1+Q)\,Q^{-\frac{nq}{2}}n^{2q}]\\&\asymp\left|t\right|^{-\frac{3q}{2}}\dsps\sum_{n=0}^{\infty}Q^{-n(\frac{q}{2}-1)}(1+n)^{2q}\\&\lesssim|t|^{-\frac{3}{2}q}.
\end{align*}
\end{proof}
Let us turn to $L^{q^{\prime}}(\mathbb{T})\to L^q(\mathbb{T})$ mapping properties of the Schr\"odinger operator $e^{it\mathcal{L}}$. 
\begin{theoreme}\label{theo dispersiv}
Let $2\pp q\le\infty.$ Then the following dispersive estimates hold~:
\begin{equation}\label{disp}
\left\|e^{it\mathcal{L}}\right\|_{L^{q^{\prime}}(\mathbb{T})\to L^{q}(\mathbb{T})}\lesssim
\begin{cases}
1 \hspace{2mm}\qquad\textrm{if}\;\;|t|<1,
\\|t|^{-\frac{3}{2}} \quad\textrm{if}\;\;|t|\geq 1.
\end{cases}
\end{equation}

 In the case $q=2$, recall that $e^{it\mathcal{L}}$ is a one parameter group of unitary operators .
\end{theoreme}
\begin{proof}[Proof]
If $q=\infty$, we use the elementary inequality 
$$
\left\|e^{it\mathcal{L}}f\right\|_{L^{\infty}}=\left\|f*s_t\right\|_{L^{\infty}}\le\,\left\|f\right\|_{L^1}\left\|s_t\right\|_{L^{\infty}}\quad \forall\, f\in L^{1}(\mathbb{T})
$$
and Proposition $\ref{pointwise}$ or Corollary $\ref{stlq}$
to conclude.
Assume $2\pp q\pp \infty$ and let $2\pp r\pp\infty$ such that $\frac{q}{2}\pp r\pp q.$
According to Corollary $(\ref{lq estimate}),$ we have 
$$\left\|e^{it\mathcal{L}}f\right\|_{L^{q}}=\left\|f*s_t\right\|_{L^q}\lesssim\left\|f\right\|_{L^{{q}^{\prime}}}\left\|s_t\right\|_{L^r}\quad \forall f\in L^{q^{\prime}}(\mathbb{T}).$$
We conclude again using Corollary $\ref{stlq}$.
\end{proof}
Theorem $\ref{theo dispersiv}$ can be generalized as follows.
\begin{corollarye}
Let $2\pp q,\,\tilde{q},\pp\infty.$ Then
$$\left\|e^{it\mathcal{L}}\right\|_{L^{{\tilde q}^{\prime}}\to L^q}\lesssim
\begin{cases}
1\qquad\hspace{0.6cm}\textrm{if}\;\;|t|\le 1,\\
|t|^{-\frac32}\qquad\textrm{if}\;\;|t|\ge 1.
\end{cases}
$$
\end{corollarye}
\begin{proof}[Proof]
If $q=\infty$, we use the elementary inequality
$$\left\|e^{it\mathcal L}f\right\|_{L^\infty}=\left\|f*s_t\right\|_{L^\infty}\le\left\|f\right\|_{L^{{\tilde{q}}^{\prime}}}\left\|s_t\right\|_{L^{\tilde{q}}}$$
and Corollary $\ref{stlq}$ to conclude.
If $\tilde{q}=\infty$, we use instead
$$\left\|e^{it\mathcal{L}}f\right\|_{L^q}=\left\|f*s_t\right\|_{L^q}\le\left\|f\right\|_{L^1}\left\|s_t\right\|_{L^q}$$
and conclude similarly.
The general case is obtained by interpolation between these two cases and the case $q=\tilde{q}$ considered in Theorem $\ref{theo dispersiv}$.
\end{proof}
Consider next the inhomogeneous linear Schr\"odinger equation on $\mathbb{T}$~:
\begin{equation}\label{sch2}
\begin{cases}
i\partial_{t}u(t,x)+\mathcal{L}u(t,x)=F(t,x),\\
u(0,x)=f(x),
\end{cases}
\end{equation}
whose solution is given by Duhamel's formula~:
\begin{equation}\label{duhamel}
u(t,x)=e^{it\mathcal{L}}f(x)-i\int_{0}^{t}e^{i(t-s)\mathcal{L}}F(s,x)\,ds.
\end{equation}
Recall the square $\mathcal{C}=(0,\frac12]\times[0,\frac12)\cup\{(0,\frac12)\}$ 
(see Fig.2.1).
\begin{theoreme}\label{theo strichartz}
Assume that $(\frac{1}{p},\frac{1}{q})$ and $(\frac{1}{\tilde{p}},\frac{1}{\tilde{q}})$ belong to $\mathcal{C}$. Then the following Strichartz estimate holds for solutions to the Cauchy problem $(\ref{sch2})$~:
\begin{equation}\label{strichartz}
\left\|u\right\|_{L^{p}_{t}L^{q}_{x}}\lesssim\left\|f\right\|_{L^{2}_{x}}+\left\|F\right\|_{L^{\tilde{p}^{\prime}}_{t}L^{\tilde{q}^{\prime}}_{x}}.
\end{equation}
\end{theoreme}

\begin{proof}[Proof]
We proceed as in the Euclidean case (\cite{k}, \cite{gv}), or in  the hyperbolic case (\cite{ap}, \cite{is}). Consider the operator
$$Tf(t,x)=e^{it\mathcal{L}_x}f(x)$$
and its adjoint 
$$T^*F(x)=\int_{\mathbb{R}}e^{-is\mathcal{L}_x}F(s,x)\,ds.$$
The method consists in proving the $L^{p^{\prime}}_{t}L^{q^{\prime}}_{x}\to L^{p}_{t}L^{q}_{x}$ boundedness of the operator 
$$TT^*F(t,x)=\int_{\mathbb{R}}e^{i(t-s)\mathcal{L}_x}F(s,x)\,ds$$
and of its truncated version
$$\widetilde{TT^*}F(t,x)=\int_{0}^t e^{i(t-s)\mathcal{L}_x}F(s,x)\,ds,$$
for every $(\frac1p,\frac1q)\in\mathcal{C}$. The endpoint $(0,\frac{1}{2})$ is settled by $L^2$ conservation and we are left with the couples $(\frac1p,\frac1q)\in (0,\frac12]\times [0,\frac12)$. Thus we consider the operator 
$$AF(t,x)=\int_{a(t)}^{b(t)}e^{i(t-s)\mathcal{L}_x}F(s,x)\,ds,$$
where $a(t)\pp b(t)$ are suitable functions of $t\in\mathbb{R}.$
 Then
\begin{align*}
\left\|AF\right\|_{L^{p}_{t}L^{q}_{x}}&=\left\|\int_{a(t)}^{b(t)}e^{i(t-s)\mathcal{L}_x}F(s,x)\,ds\right\|_{L^{p}_{t}L^{q}_{x}}\\&\lesssim\left\|\dsps\int_{(a(t),b(t))\cap\left\{|t-s|<1\right\}}e^{i(t-s)\mathcal{L}_x}F(s,x)\,ds\right\|_{L^{p}_{t}L^{q}_{x}}\\&+\left\|\dsps\int_{(a(t),b(t))\cap\left\{|t-s|\geq 1\right\}}e^{i(t-s)\mathcal{L}_x}F(s,x)\,ds\right\|_{L^{p}_{t}L^{q}_{x}}\\&\lesssim\left\|\int_{|t-s|\pp 1}\left\|F(s,x)\right\|_{L^{q^{\prime}}_{x}}\,ds\right\|_{L^{p}_{t}}+\left\|\int_{|t-s|\geq 1}\left|t-s\right|^{-\frac{3}{2}}\left\|F(s,x)\right\|_{L^{q^{\prime}}_{x}}\right\|_{L^{p}_{t}}\,.
\end{align*}

\hspace{-0.5cm}On $\mathbb{R}$, the convolution kernels
$\1_{\{|t-s|\pp 1\}}$ and $|t-s|^{-\frac{3}{2}}\1_{\left\{|t-s|\geq 1\right\}}$ define bounded operators from 
$L^{p_1}_s$ to $L^{p_2}_t$ , for all $1\le p_1\le p_2\le\infty,$ in particular from $L^{p^{\prime}}_s$ to $L^{p}_t$, for all $2\le p\le\infty.$ By choosing suitably $a(t)$ and $b(t)$ we deduce the $L^{p^{\prime}}_t L^{q^{\prime}}_x\to L^p_t L^q_x$ boundedness of $TT^*$ and $\widetilde{TT^*}$. Indices are finally decoupled, using the $TT^*$ argument.
\end{proof}

\section{Well-posedness}
Strichartz estimates for inhomogeneous linear equations are used to prove local or global well-posedness for nonlinear problems. In this section we obtain some results along these lines 
for the Schr\"odinger equation $(\ref{semilin})$ on $\mathbb{T},$ with a power-like nonlinearity F as in $(\ref{nonlinearity}).$
Let us recall the definition of well-posedness in $L^2(\mathbb{T})$.
\begin{defne}
The NLS equation $(\ref{semilin})$ is locally well-posed in $L^2(\mathbb{T})$ if, for any bounded subset $B$ of $L^2(\mathbb{T})$ there exists $T\pg 0$ and a Banach space $X_T,$ continuously embedded into $C((-T,+T);L^2(\mathbb{T}))$, such that
\begin{itemize}
 \item for any Cauchy data $f\in B$, $(\ref{semilin})$ has a unique solution $u\in X_T$;
 \item the map $f\to u$ is continuous from $B$ to $X_T$.
\end{itemize}
The equation is globally well-posed if these properties hold with $T=\infty$.
\end{defne}
As we have obtained better estimates on $\mathbb{T}$ than on $\mathbb{R}^{n}$ or $\mathbb{H}^{n}$, we may expect a better well-posedness results. We shall prove indeed well-posedness with no restriction on $\gamma$.
\begin{theoreme}\label{well posedness}
Let $1\pp\gamma\pp 1$. Then the NLS $(\ref{semilin})$ is globally well-posed for small $L^2$ data. It is locally well-posed for arbitrary $L^{2}$ data. Moreover, under the gauge invariance condition $(\ref{gauge})$, local solutions extend to global ones.
\end{theoreme}
\begin{proof}[Proof]
We resume the standard fixed point method based on Strichartz estimates. Define $u=\Psi(v)$ as the solution to the Cauchy problem
\begin{equation}\label{vproblem}
\begin{cases}
i\partial_t u(t,x)+\,\mathcal{L}_x u(t,x)=F(v(t,x)),
\\
u(0,x)=f(x),
\end{cases}
\end{equation}
which is given by Duhamel's formula~:
$$u(t,x)=e^{it\mathcal{L}_x}f(x)-i\int_{0}^{t}e^{i(t-s)\mathcal{L}_x}F(v(s,x))\,ds.$$
According to Theorem $\ref{theo strichartz}$ we have \begin{equation}\label{estimate}
\left\|u\right\|_{L^{\infty}_{t}L^{2}_{x}}+\left\|u\right\|_{L^{p}_{t}L^{q}_{x}}\lesssim\,\left\|f\right\|_{L^{2}_{x}}+\left\|F(v)\right\|_{L^{{\tilde{p}}^{\prime}}_{t}L^{{\tilde{q}}^{\prime}}_{x}}
\end{equation}
for every $2\pp p,\,\tilde{p}\le\infty$,\;\;$2\le q,\,\tilde{q}\pp\infty$. Moreover
$$\left\|F(v)\right\|_{L^{{\tilde{p}}^{\prime}}_{t}L^{{\tilde{q}}^{\prime}}_{x}}\lesssim\left\|\,|v|^{\gamma}\,\right\|_{L^{{\tilde{p}}^{\prime}}_{t}L^{{\tilde{q}}^{\prime}}_{x}}\leq\left\|v\right\|^{\gamma}_{L^{{\tilde{p}}^{\prime}\gamma}_{t}L^{{\tilde{q}}^{\prime}\gamma}_{x}}$$
by our nonlinearity assumption $(\ref{nonlinearity})$. Thus
\begin{equation}\label{str}
\left\|u\right\|_{L^{\infty}_{t}L^{2}_{x}}+\left\|u\right\|_{L^{p}_{t}L^{q}_{x}}\lesssim\left\|f\right\|_{L^{2}_{x}}+\left\|v\right\|^{\gamma}_{L^{{\tilde{p}}^{\prime}\gamma}_{t}L^{{\tilde{q}}^{\prime}\gamma}_{x}}.
\end{equation}
In order to remain within the same function space, we require 
\begin{equation}\label{p,q}
p={\tilde{p}}^{\prime}\gamma\,,\,q\le{\tilde{q}}^{\prime}\gamma.
\end{equation}
It is clear that these conditions are fullfilled if we take for instance
$$p=q=\tilde{p}=\tilde{q}=1+\gamma.$$
For such a choice $\Psi$ maps $L^{\infty}(\mathbb{R},L^{2}(\mathbb{T}))\cap L^{p}(\mathbb{R},L^{q}(\mathbb{T}))$ into itself and actually $X=C(\mathbb{R},L^{2}(\mathbb{T}))\cap L^{p}(\mathbb{R},L^{q}(\mathbb{T})$ into itself. Since $X$ is a Banach space for the norm
$$\left\|u\right\|_{X}=\left\|u\right\|_{L^{\infty}_{t}L^{2}_{x}}+\left\|u\right\|_{L^{p}_{t}L^{q}_{x}},$$
it remains to show that $\Psi$ is a contraction in the ball 
$$X_{\varepsilon}=\left\{u\in X | \left\|u\right\|_{X}\le\varepsilon\right\},$$
provided $\varepsilon\pg 0$ and $\left\|f\right\|_{L^2}$ are sufficiently small. Let $v_1,v_2\in X$ and  $u_1=\Psi(v_1),\,u_2=\Psi(v_2)$.
Arguing as above and using in addition H\"older's inequality, we estimate
\begin{align*}
\left\|u_{1}-u_{2}\right\|_{X}&\lesssim\left\|F(v_{1})-F(v_{2})\right\|_{L^{{\tilde{p}}^{\prime}}_{t}L^{{\tilde{q}}^{\prime}}_{x}}\\&\lesssim\left\|\left\{|v_{1}|^{\gamma-1}+|v_{2}|^{\gamma-1}\right\}|v_{1}-v_{2}|\right\|_{L^{{\tilde{p}}^{\prime}}_{t}L^{{\tilde{q}}^{\prime}}_{x}}\\&\lesssim\left\{\left\|v_{1}\right\|^{\gamma-1}_{L^{p}_{t}L^{q}_{x}}+\left\|v_{2}\right\|^{\gamma-1}_{L^{p}_{t}L^{q}_{x}}\right\}\left\|v_{1}-v_{2}\right\|_{L^{p}_{t}L^{q}_{x}},
\end{align*}
hence
\begin{equation}\label{str1}
\left\|u_{1}-u_{2}\right\|_{X}\lesssim\left\{\left\|v_{1}\right\|_{X}^{\gamma-1}+\left\|v_{2}\right\|_{X}^{\gamma-1}\right\}\left\|v_{1}-v_{2}\right\|_{X}.
\end{equation}
If we assume $\left\|v_1\right\|_X\le\varepsilon\,,\left\|v_2\right\|_X\le\varepsilon\;\textrm{and}\;\left\|f\right\|_{L^2}\le\delta,$ then $(\ref{str})$ and $(\ref{str1})$ yield
$$\left\|u_{1}\right\|_{X}\le C(\delta+\varepsilon^{\gamma})\le\varepsilon\;,\left\|u_{2}\right\|_{X}\le C(\delta+\varepsilon^{\gamma})\le\varepsilon$$
and
$$\left\|u_{1}-u_{2}\right\|_{X}\le 2\,C\varepsilon^{\gamma-1}\left\|v_{1}-v_{2}\right\|_{X}\le \frac12\left\|v_1-v_2\right\|_X$$
if $\varepsilon^{\gamma-1}\le\frac{1}{4(C+1)}$ and $\delta\le\frac{3\epsilon}{4(C+1)}.$ We obtain our first conclusion by applying the fixed point theorem in the complete metric space $X_\varepsilon.$\\

Let us next show that $(\ref{semilin})$ is locally well-posed for arbitrary $L^2$ data . Consider a small interval $I=[-T,T]$. We proceed as above , except that we require $\frac{\gamma}{q}+\frac1q\pp1$ and get an additional factor $T^{\lambda}$ with $\lambda=1-\frac{\gamma}{q}-\frac{1}{\tilde{q}}\pg0,$ by applying H\"older's inequality in time. This way we obtain the estimates 
\begin{equation}\label{local}
\left\|u\right\|_{X}\le C(\left\|f\right\|_{L^{2}}+T^{\lambda}\left\|v\right\|^{\gamma}_{X}),
\end{equation} 
and
\begin{equation}
\left\|u_{1}-u_{2}\right\|_{X}\le C T^{\lambda}(\left\|v_{1}\right\|_{X}^{\gamma-1}+\left\|v_{2}\right\|_{X}^{\gamma-1})\left\|v_{1}-v_{2}\right\|_{X}
\end{equation} 
where $X=C(I;L^{2}(\mathbb{T}))\cap L^{p}(I,L^{q}(\mathbb{T}))$ and $p\in 
[2,\infty),\;q\in(2,\infty)$ satisfy $p\pg\gamma,\;q\pp 2\gamma.$\\
$\,$

As a consequence, we deduce that $\Psi$ is a contraction in the ball
$$X_M=\left\{u\in X|\,\left\|u\right\|_X\le M\,\right\},$$
provided $M\pg0$ is large enough respectively $T\pg 0$ is small enough, more precisely
 \begin{equation}\label{T condition}
C\left\|f\right\|_{L^2}\le \frac{3}{4} M\;\; \textrm{and}\;C T^\lambda M^{\gamma -1}\le\frac{1}{4}\,.
\end{equation}
We conclude as before.
Notice that according to $(\ref{T condition})$, T depends only on the $L^2$ norm of the initial data~:
$$T^{\lambda}\lesssim \left\|f\right\|^{-\frac{\gamma-1}{\lambda}}_{L^2}\,.$$
Thus if the nonlinearity $F$ is gauge invariant as in $(\ref{gauge}),$ then $L^2$ conservation allows us to iterate and deduce global from local existence, for arbitrary data $f\in L^2$. 
\end{proof}

\section{Scattering}

Consider still the NLS $(\ref{semilin})$ with a powerlike nonlinearity $F$ as in $(\ref{nonlinearity})$

\begin{theoreme}\label{scattering}
Assume that $1\pp\gamma\pp\infty.$ Then, under the gauge invariance condition $(\ref{gauge})$, for any data $f\in L^2(\mathbb{T}),$ the unique global solution $u$ provided by Theorem $\ref{well posedness}$ scatters to a linear solution, 
 that is there exist $u_\pm\in L^2$ such that
$$\left\|u(t)-e^{it\mathcal{L}}u_\pm\right\|_{L^{2}(\mathbb{T})}\to 0 \;\textrm{as}\;t\to\pm\infty.$$
\end{theoreme}
\begin{proof}[Proof]
We will use the following Cauchy criterion~:
If $\left\|z(t_1)-z(t_2)\right\|_{L^2} \to0$ as $t_1,\,t_2,$ tend both to $\pm\infty$, then there exists $z_\pm \in L^2(\mathbb{T})$ such that $\left\|z(t)-z_\pm\right\|_{L^2}\to 0$ as $t\to\pm\infty$. According to Theorem $(\ref{theo strichartz})$ we have, for all $t_1\le t_2$
 
\begin{align*}
\left\|e^{-it_2\mathcal{L}}u(t_2)-e^{-it_1\mathcal{L}}u(t_1)\right\|_{L^2(\mathbb{T})}&=\left\|\int_{t_1}^{t_2}e^{-is\mathcal{L}}F(u(s))\,ds\right\|_{L^{2}(\mathbb{T})}\\&\lesssim\left\|u\right\|^{\gamma}_{L^{p}([t_1,t_2],L^{q}(\mathbb{T}))}.
\end{align*}
Since $u\in L^{p}(\mathbb{R},L^{q}(\mathbb{T}))$, the last expression vanishes as $t_1\le t_2$ tend both to $-\infty$ or to $+\infty$. Thus, using the Cauchy criterion above, one gets the desired result.
\end{proof}
\textbf{Remark}. If the nonlinearity is not gauge invariant we will still have scattering for small $L^2$ for all $1\pp\gamma\pp \infty$.

\end{document}